\documentclass[11pt,a4paper]{article}
\usepackage[T1]{fontenc}
\usepackage[utf8]{inputenc}
\usepackage{lmodern,amsmath,amssymb,amsthm}
\usepackage[margin=1in]{geometry}
\usepackage{microtype}
\usepackage[colorlinks=true,linkcolor=blue,citecolor=blue,urlcolor=blue]{hyperref}
\newtheorem{theorem}{Theorem}
\newtheorem{lemma}[theorem]{Lemma}
\newtheorem{proposition}[theorem]{Proposition}
\theoremstyle{remark}
\newcommand{\Z}{\mathbb Z}
\newcommand{\C}{\mathbb C}
\newcommand{\PP}{\mathbb P}
\DeclareMathOperator{\Tor}{Tor}
\DeclareMathOperator{\Bl}{Bl}
\title{Torsion Cycles in Projective Threefolds:\\\large From Serre's Construction to Artin--Mumford}
\author{Alberto Verjovsky}
\date{}
\begin{document}
\maketitle
\begin{center}
\textit{To my teacher Solomon Lefschetz}
\end{center}
\medskip
\begin{abstract}
The torsion subgroup of third integral cohomology is a birational invariant of smooth projective complex threefolds; its nonvanishing obstructs stable rationality. We recall a proof using Hironaka's elimination of indeterminacies and the integral blowup formula. The Godeaux--Serre construction and Poincar\'e duality imply that every finite abelian group is the third integral cohomology group of a smooth projective complex threefold. This classical realization also has a description through families of concurrent plane lines. Their incidence threefolds have smooth normalizations, obtained by an explicit blowup and described as projective-line bundles over smooth surfaces. The Lefschetz theorem determines the fundamental group of these surfaces, and the projective bundle formula computes the third cohomology. The resulting threefolds are birational to products of surfaces with the projective line, and their torsion three-cycles are represented by sphere bundles over loops. Examples and the comparison with the unirational Artin--Mumford threefold explain the role of the base surface and the distinction between smooth blowups and resolution of a singular space.
\end{abstract}

\medskip
\noindent\textbf{Keywords:} Projective threefolds; torsion in integral cohomology; birational invariants; Godeaux--Serre construction; Artin--Mumford example; concurrent plane lines.

\smallskip
\noindent\textbf{2020 Mathematics Subject Classification:} Primary 14F25; Secondary 14E08, 14J30, 14E15.

\bigskip

All varieties are over $\C$, and homology and cohomology refer to the analytic topology with integral coefficients. A threefold has complex dimension three. For an abelian group $A$, its torsion subgroup is $\Tor A=\{a\in A:na=0\text{ for some }n\geq1\}$.

\section{The classical realization from surfaces}

This paper is expository. Its starting point is that the torsion subgroup of $H^3$ is a birational invariant of smooth projective complex threefolds. Since it vanishes for projective space, nonzero torsion obstructs rationality, and indeed stable rationality. We begin with a proof using Hironaka's elimination of indeterminacies and integral topology, as in Artin and Mumford \cite[\S1, Proposition 1]{AM}.

Its particular emphasis is an explicit construction from families of concurrent plane lines, together with a resolution whose geometry makes the cohomology calculation transparent. The underlying existence result is a classical consequence of the Godeaux--Serre construction, Poincar\'e duality, and the K\"unneth formula. The incidence construction relates this result to the geometry of singular plane curves and to the distinction between smooth blowups and resolution.

Artin and Mumford's example \cite{AM} is a smooth projective unirational threefold whose third integral cohomology has nonzero torsion. The simultaneous occurrence of unirationality and torsion is essential to that example. Without the unirationality requirement, realization of a prescribed finite abelian group follows from the Godeaux--Serre construction and elementary topology.

\begin{proposition}[Birational invariance]\label{torsion-birational}
If $Y$ and $Y'$ are birational smooth connected projective complex threefolds, then
\[
 \Tor H^3(Y,\Z)\cong\Tor H^3(Y',\Z).
\]
The same assertion holds in every dimension and is invariant under taking a product with projective space. In particular, a stably rational smooth projective variety has torsion-free third integral cohomology.
\end{proposition}
\begin{proof}
First consider the blowup $b:\widetilde Y=\Bl_ZY\to Y$ along a smooth center of codimension $c\geq2$. The integral blowup formula is
\[
 H^k(\widetilde Y,\Z)
 \cong H^k(Y,\Z)\oplus
 \bigoplus_{j=1}^{c-1}H^{k-2j}(Z,\Z),
\]
with $b^*$ identifying the first summand. This formula follows from the Thom isomorphism and the projective bundle formula for the exceptional divisor $\PP(N_{Z/Y})$; all normal bundles involved are complex, hence integrally oriented. In degree three only $j=1$ contributes, so
\[
 H^3(\widetilde Y,\Z)\cong H^3(Y,\Z)\oplus H^1(Z,\Z).
\]
The universal coefficient theorem gives
$H^1(Z,\Z)=\operatorname{Hom}(H_1(Z,\Z),\Z)$, which is torsion-free. Consequently $b^*$ is an isomorphism on torsion. In dimension three the nontrivial centers are points and smooth curves: a point adds nothing in degree three, and a curve adds its free group $H^1$. Blowing up a smooth divisor is an isomorphism.

Now let $\phi:Y\dashrightarrow Y'$ be birational. By Hironaka's elimination of indeterminacies \cite{Hironaka}, there is a smooth projective variety $W$ and morphisms
\[
 Y\xleftarrow{\ b\ }W\xrightarrow{\ f\ }Y',
 \qquad f=\phi\circ b,
\]
where $b$ is a sequence of blowups along smooth centers and $f$ is birational. Thus $b^*$ identifies $\Tor H^3(Y,\Z)$ with $\Tor H^3(W,\Z)$.

The map $f$ has topological degree one. To see why its pullback is injective even on torsion, let $D_W,D_{Y'}$ denote integral Poincar\'e duality and define the cohomological pushforward by
\[
 f_!=D_{Y'}^{-1}\circ f_*\circ D_W.
\]
For every integral cohomology class $a$, naturality of the cap product and $f_*[W]=[Y']$ give
\[
 f_!(f^*a)
 =D_{Y'}^{-1}\bigl(f_*(f^*a\cap[W])\bigr)
 =D_{Y'}^{-1}(a\cap[Y'])=a.
\]
Hence $f^*$ injects $\Tor H^3(Y',\Z)$ into $\Tor H^3(W,\Z)$, and therefore into a group isomorphic to $\Tor H^3(Y,\Z)$. Applying the same argument to $\phi^{-1}$ gives an injection in the reverse direction. Both torsion groups are finite, since smooth projective varieties have finitely generated integral cohomology. Their orders are consequently equal, and either injection is an isomorphism.

Nothing in this argument depends on dimension three. For $r\geq1$, the integral K\"unneth formula gives
\[
 H^3(Y\times\PP^r,\Z)
 \cong H^3(Y,\Z)\oplus H^1(Y,\Z),
\]
and again the second summand is torsion-free. Thus the invariant is unchanged by products with projective space and by birational equivalence. If $Y$ is stably rational, some $Y\times\PP^r$ is birational to a projective space, whose third cohomology is zero. This proves the last assertion.
\end{proof}

We now turn to the realization of prescribed torsion groups.

\begin{lemma}[Godeaux--Serre]\label{GS}
Every finite group is the fundamental group of a smooth connected projective complex surface.
\end{lemma}

A construction and proof are given in \cite[\S2.4, Theorem 2.4.1 and footnote 6]{Achinger}; see also \cite[\S16]{Serre}. To recall the construction, take enough copies of the regular representation that the locus with nontrivial projective stabilizer has codimension greater than two. In the finite quotient of projective space, a general complete intersection of sufficiently ample hypersurfaces of dimension two avoids the image of this locus. Its inverse image is a smooth complete-intersection surface with a free group action. The Lefschetz theorem makes this surface simply connected. Its quotient is the required surface.

\begin{proposition}\label{product}
For every finite abelian group $G$, there is a smooth connected projective complex threefold $N$ with $H^3(N,\Z)\cong G$.
\end{proposition}
\begin{proof}
Choose a surface $S$ with $\pi_1(S)\cong G$ by Lemma~\ref{GS}. Since $G$ is abelian, the first homology group is $H_1(S,\Z)\cong G$. Integral Poincar\'e duality for the compact oriented real four-manifold $S$, and the universal coefficient theorem, give
\[
 H^3(S,\Z)\cong H_1(S,\Z)\cong G,
 \qquad H^1(S,\Z)=\operatorname{Hom}(G,\Z)=0.
\]
Set $N=S\times\PP^1$. The integral cohomology of $\PP^1$ is free, so the K\"unneth formula gives
\[
 H^3(N,\Z)\cong H^3(S,\Z)\oplus H^1(S,\Z)\cong G.
\]
\end{proof}

The same classical ingredients give the following realization in the universal family of plane curves. A binary form is a homogeneous polynomial in two variables. A resolution of an integral variety $X$ is a proper birational morphism from a smooth variety to $X$. The strict transform under a blowup is the closure of the inverse image of the complement of the center.

\begin{theorem}\label{main}
Let $G$ be a finite abelian group. There exist an integer $d\geq2$, a smooth connected projective surface
\[
 S\subset\PP H^0(\PP^2,\mathcal O_{\PP^2}(d)),
\]
and a point $q\in\PP^2$ with the following properties:
\begin{enumerate}
\item Every curve parametrized by $S$ is a union of $d$ lines through $q$, counted with multiplicity, and the general curve consists of distinct lines.
\item The incidence threefold
\[
 X=\{(s,z)\in S\times\PP^2:z\in C_s\}
\]
is integral. Its strict transform $M$ in $S\times\Bl_q\PP^2$ is smooth, and the induced morphism $M\to X$ is a projective resolution.
\item One has
\[
 H^3(M,\Z)\cong G,\qquad \pi_1(M)\cong G.
\]
\end{enumerate}
If $G\ne0$, then $M$ is neither stably rational nor unirational.
\end{theorem}

The integer $d$ is allowed to depend on $G$. The assertion does not prescribe the group for each fixed degree.

\section{The surface of line choices}

Choose a surface $S_0$ with $\pi_1(S_0)\cong G$ by Lemma~\ref{GS}, and choose a very ample line bundle $A$ on $S_0$. Thus the complete linear system of $A$ embeds $S_0$ in projective space. Put $r+1=h^0(S_0,A)$, and choose $d\geq\max\{2,r\}$.

Fix $q=[0:0:1]\in\PP^2$. The lines through $q$ are parametrized by $\PP^1$: the point $[u:v]$ corresponds to the line
\[
 \ell_{[u:v]}=\{[x:y:z]:vx-uy=0\}.
\]
A nonzero homogeneous binary form of degree $d$ determines a union of $d$ such lines. These forms, modulo scalars, constitute a linear subspace $\PP^d$ of the parameter space of degree-$d$ plane curves.

On $S_0\times\PP^1$ choose a general section
\[
 F\in H^0(S_0,A)\otimes H^0(\PP^1,\mathcal O_{\PP^1}(d)).
\]
Writing
\begin{equation}\label{binary}
 F(s;u,v)=\sum_{i=0}^d a_i(s)u^iv^{d-i},
\end{equation}
the coefficients $a_0,\ldots,a_d$ span $H^0(S_0,A)$ for a general choice, since $d+1\geq h^0(S_0,A)$. They have no common zero and define a closed embedding
\[
 \iota:S_0\longrightarrow\PP^d,\qquad
 s\longmapsto[F(s;u,v)].
\]
Indeed, this is the embedding by the complete linear system of $A$, followed by a linear embedding of projective spaces. Set $S=\iota(S_0)$ and identify $S$ with $S_0$.

Define the surface of line choices by
\begin{equation}\label{Tdef}
 T=\{(s,[u:v])\in S\times\PP^1:F(s;u,v)=0\}.
\end{equation}
The line bundle $A\boxtimes\mathcal O_{\PP^1}(d)$, meaning the tensor product of the two pulled-back line bundles, is very ample. Bertini's theorem implies that $T$ is smooth for general $F$. The conditions of smoothness and coefficient spanning are nonempty Zariski-open conditions in the same vector space of sections, and therefore hold simultaneously.

The projection $T\to S$ is finite of degree $d$: every fiber is the zero scheme of a nonzero binary form of degree $d$. For general $F$, the general form $F(s;u,v)$ has distinct roots. To see that this is a nonempty open condition, fix one point of $S$ and use surjectivity of evaluation of the globally generated line bundle $A$ to prescribe any square-free binary form there.

The form of the Lefschetz theorem used here states that, for a smooth connected projective variety $Y$ of complex dimension $m\geq3$ and a smooth ample divisor $D$, inclusion induces an isomorphism on $\pi_1$ and on integral homology in degrees $i<m-1$. Applying this to $T\subset S\times\PP^1$ gives
\begin{equation}\label{lefschetz}
 \pi_1(T)\cong\pi_1(S\times\PP^1)\cong G,
 \qquad H_1(T,\Z)\cong G.
\end{equation}
In particular, $T$ is connected, hence integral. The homological and homotopical forms of the theorem are stated in \cite[Theorem 2.3.1]{Achinger}.

\section{Resolution and cohomology}

Replace $u,v$ in \eqref{binary} by $x,y$ and define
\[
 X=\{(s,[x:y:z])\in S\times\PP^2:F(s;x,y)=0\}.
\]
This is the restriction of the universal degree-$d$ plane curve to $S$. Every fiber consists of lines through $q$. Since $d\geq2$, all first derivatives of the equation vanish along $S\times\{q\}$, so $X$ is singular there.

The blowup $B=\Bl_q\PP^2$ has the incidence description
\[
 B=\{([x:y:z],[u:v])\in\PP^2\times\PP^1:xv=yu\}.
\]
The projection $\rho:B\to\PP^1$ is a projective-line bundle: its fiber over $[u:v]$ is the line $\ell_{[u:v]}$, with its direction retained at $q$. Here a projective-line bundle is the projectivization of a rank-two algebraic vector bundle.

Let
\begin{equation}\label{Mdef}
 M=T\times_{\PP^1}B.
\end{equation}
Then $M\to T$ is a projective-line bundle. Consequently $M$ is smooth, connected, and projective. Forgetting the direction defines a projective morphism $M\to X$, which is an isomorphism away from $S\times\{q\}$: a point $[x:y:z]\ne q$ determines its unique direction $[x:y]$.

This morphism identifies $M$ with the strict transform of $X$. In a blowup chart write $(x,y)=(t,tw)$; the pulled-back equation becomes
\[
 F(s;t,tw)=t^dF(s;1,w).
\]
Removing the exceptional factor $t^d$ leaves precisely the equation of \eqref{Mdef}. The other chart is identical. The complement in $M$ of the section lying over $S\times\{q\}$ is integral and dense. Moreover, $X$ is a hypersurface in a smooth variety, has no component contained in the codimension-two subset $S\times\{q\}$ of $S\times\PP^2$, and is generically reduced. A hypersurface in a smooth variety has no embedded associated components, so generic reducedness implies reducedness here. Thus $X$ is integral, and $M\to X$ is a resolution.

\subsection*{The resolution is the normalization}
An integral variety is normal if its local rings are integrally closed in its function field. Its normalization is obtained by taking these integral closures on affine open sets; for a complex algebraic variety the normalization map is finite and birational. Smooth varieties are normal.

Write $\nu:M\to X$ for the morphism above. Away from the concurrency section it is an isomorphism. Over $(s,q)$ its fiber is the finite scheme of directions
\[
 \{[u:v]\in\PP^1:F(s;u,v)=0\}.
\]
Thus $\nu$ is proper with finite fibers, hence finite \cite[Tag 02LS]{Stacks}. Since $M$ is normal and $\nu$ is birational, $\nu$ is the normalization of $X$ \cite[Tag 0BXQ]{Stacks}. For completeness, on an affine open subset of $X$ write the finite map as $A\subset B\subset K$, where $K$ is their common function field. The ring $B$ is integral over $A$ and integrally closed in $K$. Every element of $K$ integral over $A$ is also integral over $B$, and therefore belongs to $B$. Hence $B$ is exactly the integral closure of $A$.

For a parameter with $d$ distinct directions, the normalization has $d$ distinct points over $(s,q)$. Analytically near such a parameter, the roots can be labeled and the incidence equation factors into $d$ smooth branches meeting along the concurrency section. Their labels can be permuted on continuation around loops in the parameter surface. This explains how the total algebraic variety $X$ can be integral while these local analytic branches, and its general fibers, are reducible.

The section $T\subset M$ is the intersection with the exceptional divisor of the ambient blowup. Its map to $S\times\{q\}$ is finite of degree $d$. In particular, this map separates branches without contracting a positive-dimensional fiber. This geometry accounts for the difference from a blowup of a smooth threefold along a smooth curve.

\subsection*{The cohomology calculation}
Let $p:M\to T$ be the bundle projection and let $\xi=c_1(\mathcal O_{M/T}(1))$. The integral projective bundle formula gives
\[
 H^3(M,\Z)=p^*H^3(T,\Z)\oplus
 \xi\smile p^*H^1(T,\Z).
\]
By \eqref{lefschetz} and the universal coefficient theorem,
\[
 H^1(T,\Z)=\operatorname{Hom}(H_1(T,\Z),\Z)=0.
\]
Integral Poincar\'e duality on the compact oriented real four-manifold $T$ gives
\[
 H^3(T,\Z)\cong H_1(T,\Z)\cong G.
\]
It follows that $H^3(M,\Z)\cong G$. The homotopy exact sequence of the bundle, whose fiber is simply connected, gives $\pi_1(M)\cong\pi_1(T)\cong G$.

The group $\Tor H^3$ is a stable birational invariant of smooth projective complex varieties and vanishes for projective space \cite{AM}. Here stable rationality means that $M\times\PP^m$ is birational to projective space for some $m\geq0$. Thus $G\ne0$ obstructs stable rationality. Unirationality means the existence of a dominant rational map from projective space. A smooth projective complex unirational variety is simply connected \cite{SerreUnirational}; see also \cite[Chapter IV]{Kollar}. Since $\pi_1(M)\cong G$, nonzero $G$ also excludes unirationality. This completes the proof of Theorem~\ref{main}.

\section{Birational description and examples}

\begin{proposition}\label{birational}
The threefold $M$ of Theorem~\ref{main} is birational to $T\times\PP^1$. Consequently, $M$ is stably rational if and only if $T$ is stably rational, and $M$ is unirational if and only if $T$ is unirational.
\end{proposition}
\begin{proof}
Write $M=\PP_T(E)$ for a rank-two vector bundle $E$ on $T$. There is a nonempty Zariski-open subset $U\subset T$ on which $E$ is trivial. Thus
\[
 p^{-1}(U)\cong U\times\PP^1,
 \qquad \C(M)\cong\C(T)(t),
\]
where $t$ is transcendental over the function field $\C(T)$. This proves the birational assertion. Stable rationality is equivalent to the function field becoming purely transcendental after adjoining finitely many independent variables; the displayed equality proves the equivalence for $M$ and $T$. If $M$ is unirational, composition with $p$ shows that $T$ is unirational. Conversely, a dominant rational parametrization of $T$, together with one independent parameter, dominates $T\times\PP^1$ and hence $M$.
\end{proof}

The torsion is therefore inherited from the surface $T$. Since $H_1(T,\Z)\cong G$, the surface itself has $H^3(T,\Z)\cong G$. The family of plane curves provides an incidence description of the threefold and a direct resolution of its singularities.

\subsection*{A rational example}
Take $S=\PP^2$ with homogeneous coordinates $[a:b:c]$, $A=\mathcal O_{\PP^2}(1)$, and $d=2$. Set
\[
 F([a:b:c];u,v)=au^2+buv+cv^2.
\]
For each $[u:v]\in\PP^1$, this is a nonzero linear equation in $[a:b:c]$. Therefore $T\to\PP^1$ is a projective-line bundle. More explicitly, $T$ is the projectivization of the kernel of the surjective vector-bundle map
\[
 \mathcal O_{\PP^1}^{\oplus3}\longrightarrow\mathcal O_{\PP^1}(2),
 \qquad(a,b,c)\longmapsto au^2+buv+cv^2.
\]
It follows that $T$ is smooth and rational, and Proposition~\ref{birational} makes $M$ rational. Here $H^3(M,\Z)=0$. This illustrates that the singular concurrency section alone does not produce nonzero torsion. Conversely, the vanishing of this torsion is not a rationality criterion: smooth cubic threefolds have torsion-free $H^3$ and are not rational \cite{CG}. Indeed, homological Lefschetz gives $H_2(V,\Z)\cong H_2(\PP^4,\Z)\cong\Z$ for a smooth cubic threefold $V$, and the universal coefficient theorem gives $\Tor H^3(V,\Z)=0$.

\subsection*{An example with group $\Z/5\Z$}
Let
\[
 V=\{x_0^5+x_1^5+x_2^5+x_3^5=0\}\subset\PP^3,
\]
and let a primitive fifth root of unity $\zeta$ act by
\[
 [x_0:x_1:x_2:x_3]\longmapsto
 [x_0:\zeta x_1:\zeta^2x_2:\zeta^3x_3].
\]
The hypersurface $V$ is smooth. Each nonidentity group element has four distinct eigenvalues, so its fixed points in $\PP^3$ are the coordinate points; none lies on $V$. The action on $V$ is therefore free. Lefschetz gives $\pi_1(V)=0$, and the smooth projective quotient $S=V/(\Z/5\Z)$ has $\pi_1(S)\cong\Z/5\Z$. This is the classical Godeaux surface. Choosing a very ample line bundle on $S$ and applying the preceding construction gives
\[
 H^3(M,\Z)\cong\Z/5\Z,\qquad
 \pi_1(M)\cong\Z/5\Z.
\]
The free quotient exhibits the surface carrying the torsion explicitly; the embedding by a very ample line bundle determines the required degree of the plane-curve family.

\section{Blowups, resolutions, and the Artin--Mumford example}

Proposition~\ref{torsion-birational} proves invariance under blowups of a smooth variety along smooth centers and, by elimination of indeterminacies, under birational equivalence of smooth projective varieties. These smoothness hypotheses matter. For a resolution of a singular variety, Hironaka's theorem \cite{Hironaka} supplies smoothness but does not identify the torsion of the singular space with that of its resolution. In the construction above, the required group is computed directly on the explicit resolution $M$.

For the incidence threefold, the inverse image of the concurrency section is naturally $T$: a point of $T$ records a parameter and one line through $q$. The restriction of the normalization to this section is the finite degree-$d$ map $T\to S$. Thus the resolution separates the directions meeting at $q$. Its cohomology is determined by the surface $T$, through the projective bundle formula, rather than by applying the smooth blowup formula to the singular space $X$.

\section{Serre's realization and the Artin--Mumford obstruction}

Serre's realization theorem prescribes a finite fundamental group. For an abelian group $G$, Poincar\'e duality on a surface transfers this information to third cohomology:
\[
 \pi_1(S)\cong G
 \quad\Longrightarrow\quad
 H^3(S,\Z)\cong H_1(S,\Z)\cong G.
\]
Taking $S\times\PP^1$, or using the concurrent-line construction, gives a threefold with third cohomology $G$ and fundamental group $G$. For $G\ne0$, the fundamental group already rules out unirationality by Serre's theorem \cite{SerreUnirational}.

The Artin--Mumford threefold $V$ satisfies three different conditions simultaneously:
\[
 V\text{ is unirational},\qquad
 \pi_1(V)=0,\qquad
 \Tor H^3(V,\Z)\ne0.
\]
The last condition obstructs stable rationality. Thus the example distinguishes rationality from unirationality among simply connected smooth projective threefolds. The surface realization theorem does not imply this assertion.

The universal coefficient theorem describes the topological distinction. For a compact smooth manifold $Y$ it gives an exact sequence
\[
 0\longrightarrow
 \operatorname{Ext}^1_{\Z}(H_2(Y,\Z),\Z)
 \longrightarrow H^3(Y,\Z)
 \longrightarrow\operatorname{Hom}(H_3(Y,\Z),\Z)
 \longrightarrow0.
\]
Here $\operatorname{Ext}^1_{\Z}(A,\Z)$ is the group of equivalence classes of abelian-group extensions of $A$ by $\Z$. For finitely generated $A$, it is a finite group, naturally isomorphic to $\operatorname{Hom}(\Tor A,\mathbb Q/\Z)$. Since the right-hand term in the sequence is free, the sequence identifies
\[
 \Tor H^3(Y,\Z)
 \cong\operatorname{Ext}^1_{\Z}(H_2(Y,\Z),\Z).
\]
Simple connectedness forces $H_1(Y,\Z)=0$, but does not force $H_2(Y,\Z)$ to be torsion-free. In complex dimension two, integral Poincar\'e duality identifies $H^3$ with $H_1$. In complex dimension three, it identifies $H^3$ with $H_3$ instead. Accordingly, a simply connected threefold can have nonzero third-cohomology torsion arising from torsion in $H_2$. The Artin--Mumford example realizes this possibility under the geometric restriction of unirationality.

For $G=\Z/2\Z$, Theorem~\ref{main} gives a threefold with the same third-cohomology torsion group as the classical Artin--Mumford example, but with fundamental group $\Z/2\Z$. Its birational description $T\times\PP^1$ shows that the obstruction is inherited from a surface. These two constructions locate torsion in different geometric settings.

\section{Explicit three-cycles and pairs of lines}

\subsection*{Three-cycles in the concurrent-line construction}
Let $p:M\to T$ be the projective-line bundle of Theorem~\ref{main}. Choose a smooth embedded oriented loop $\gamma\subset T$ representing an element of order $n$ in $H_1(T,\Z)$. Every element of this first homology group has such a representative: it is represented by a loop, and general position in the real four-manifold $T$ removes self-intersections. The inverse image
\[
 Z_\gamma=p^{-1}(\gamma)
\]
is a closed oriented real three-manifold in $M$. The complex vector bundle $E|_\gamma$ is trivial, since every complex vector bundle over a circle is trivial. Thus
\[
 Z_\gamma\cong S^1\times\PP^1\cong S^1\times S^2.
\]
These are topological three-cycles; their real dimension is three.

\begin{proposition}
The assignment $[\gamma]\mapsto[Z_\gamma]$ induces an isomorphism
\[
 H_1(T,\Z)\xrightarrow{\ \sim\ }H_3(M,\Z).
\]
In particular, $[Z_\gamma]$ has exactly the same order as $[\gamma]$.
\end{proposition}
\begin{proof}
For an oriented compact manifold $Y$, write $D_Y$ for the Poincar\'e duality isomorphism from cohomology to homology, given by cap product with its fundamental class. Define
\[
 p^!=D_M\circ p^*\circ D_T^{-1}:H_1(T,\Z)\longrightarrow H_3(M,\Z).
\]
This is the degree-two homological transfer associated with the oriented real two-dimensional fibers. Naturality of Poincar\'e duality for inverse images under a submersion gives $p^![\gamma]=[p^{-1}(\gamma)]$, with the induced orientation. The projective bundle formula and $H^1(T,\Z)=0$ show that $p^*:H^3(T,\Z)\to H^3(M,\Z)$ is an isomorphism. Therefore $p^!$ is an isomorphism, including on torsion.
\end{proof}

The relation $n[Z_\gamma]=0$ also has a geometric bounding chain. Since $n[\gamma]=0$, there is a compact oriented surface $W$ and a map $f:W\to T$ whose boundary is the union of $n$ oriented circles, each mapped as $\gamma$. One construction uses the fact that a loop with trivial image in the abelianization of $\pi_1(T)$ is a product of commutators; a surface with handles and boundary realizes this relation. Choose $f$ smooth, keeping its boundary maps fixed. The pullback of $M$ to $W$ is an oriented real four-manifold with boundary, mapping to $M$. Its boundary consists of $n$ copies of $Z_\gamma$. Its pushed-forward fundamental chain $B$ therefore satisfies
\[
 \partial B=n Z_\gamma.
\]
The transfer isomorphism proves that no smaller positive multiple bounds when $[\gamma]$ has order $n$. Choosing loops representing the cyclic generators of $G$ gives geometric representatives for all the torsion generators in $H_3(M,\Z)$, and Poincar\'e duality gives the corresponding classes in $H^3(M,\Z)$.

\subsection*{The two-line fibers of a conic bundle}
A conic bundle over a smooth surface $B$ is a flat projective morphism whose fibers are plane conics, described locally in a projective-plane bundle by homogeneous quadratic equations. The rank of a fiber is the rank of the symmetric matrix of its quadratic equation. Its discriminant $\Delta\subset B$ is the locus of singular fibers. At a rank-two fiber the conic is the union of two distinct lines meeting at one point. Over the open subset $\Delta^\circ$ of such fibers, the choices of a line form a double cover
\[
 \widetilde\Delta^\circ
 =\{(b,\ell):b\in\Delta^\circ,\ \ell\text{ is a component of the fiber over }b\}
 \longrightarrow\Delta^\circ.
\]
Continuation along a loop can preserve the two lines individually or interchange them; this permutation is the monodromy of the double cover.

Locally along $\Delta^\circ$, the intersection point of the lines can be moved by coordinates to $[0:0:1]$. The singular conics then have equations
\[
 a(s)x^2+b(s)xy+c(s)y^2=0.
\]
This is precisely the binary-quadratic, or two-line, geometry used in our incidence construction. A transversal deformation such as
\[
 xy=t z^2
\]
shows the further geometry of a conic bundle: the fiber at $t=0$ consists of two lines, whereas the fibers at $t\ne0$ are smooth conics. The surface of parameters in our construction lies entirely in the locus of reducible curves; a conic bundle with smooth general fiber meets that locus along its discriminant curve.

Thus the two-line part of the construction describes the singular fibers and their component choices. It does not identify the whole Artin--Mumford threefold with the case $d=2$ of Theorem~\ref{main}. Indeed, with the concurrency point fixed, the parameter space for binary quadratics is $\PP^2$. A closed integral surface embedded in it is $\PP^2$ itself, giving the rational example of Section~4. In the Artin--Mumford setting the general fiber is smooth, and the topology of the discriminant and its component covers is essential.

Artin and Mumford give explicit cycles in \cite[\S2, pp.~81--84]{AM}: transporting a line component around a loop that preserves it gives a three-cycle $\beta$. Their construction also supplies a two-cycle $\alpha$ and chains $\Lambda,U$ with
\[
 \partial\Lambda=2\alpha,\qquad \partial U=2\beta,
 \qquad \Lambda\cdot\beta=1.
\]
The resulting torsion linking number is
\[
 \lambda([\alpha],[\beta])
 =\frac{\Lambda\cdot\beta}{2}\pmod{\Z}
 =\frac12\in\mathbb Q/\Z,
\]
so both classes have exact order two. Here the dot denotes oriented intersection after making the chains transverse. The common geometric operation is a line swept along a loop. In our bundle it is the inverse image of a loop in $T$; in their threefold the relations and nonvanishing are established using the degenerating conics and the linking pairing.

\section{Historical context}

The Symposium Internacional de Topolog\'ia Algebraica took place at Ciudad Universitaria in Mexico City in August 1956. Its proceedings were published by UNAM and UNESCO in 1958; the dates of the meeting and publication are therefore different \cite{UNAMHistory}. Serre's article, \emph{Sur la topologie des vari\'et\'es alg\'ebriques en caract\'eristique $p$}, appeared in that volume, on pages 24--53. Section 16 is a classical reference for the free finite-group quotient construction used in the Godeaux--Serre method. The complex topological formulation used here, including its realization by surfaces, is stated and proved in \cite[\S2.4]{Achinger}.

Serre's 1959 paper \emph{On the fundamental group of a unirational variety} concerns a different question \cite{SerreUnirational}. It proves simple connectedness for smooth projective complex unirational varieties. Thus Serre's work supplies both the construction of projective varieties with prescribed finite fundamental group and a restriction excluding nontrivial such groups in the unirational case.

Hironaka's resolution theorem appeared in 1964 \cite{Hironaka}. In characteristic zero it provides the smooth projective models needed for comparing birational invariants. Artin and Mumford use resolution in their proof of the birational invariance of third-cohomology torsion \cite[\S1, Proposition 1]{AM}. Its stable birational invariance follows by also applying the K\"unneth formula to products with projective space: the only additional contribution in degree three is $H^1$, which is torsion-free.

The rationality problem for threefolds developed rapidly in 1971--1972. Iskovskikh and Manin proved nonrationality of smooth quartic threefolds in 1971 \cite{IM}. Clemens and Griffiths proved nonrationality of smooth cubic threefolds in their 1972 paper, using the principally polarized intermediate Jacobian \cite{CG}. For a smooth cubic threefold $V$, this is the complex torus
\[
 J(V)=\frac{H^3(V,\C)}{F^2H^3(V,\C)+H^3(V,\Z)},
 \qquad F^2H^3=H^{3,0}\oplus H^{2,1},
\]
where the integral group is understood through its image in complex cohomology and $H^{p,q}$ denotes the Hodge decomposition. The intersection pairing supplies the principal polarization. Their argument uses this polarized structure to obstruct rationality; their Appendix B discusses unirationality.

Artin and Mumford's paper was received on November 8, 1971, and published in 1972 \cite{AM}. In its introduction they credit C.~P.~Ramanujam with suggesting torsion in $H^3$ as an obstruction. Their construction produces unirational examples with nonzero $2$-torsion. It supplies a topological obstruction beyond the fundamental group, in a setting where Serre's simple-connectedness theorem applies.

\subsection*{Funding}
A.V. was supported by Proyecto PAPIIT, Direcci\'on General de Asuntos del Personal Acad\'emico, Universidad Nacional Aut\'onoma de M\'exico [grant number IN103324].

\medskip
\noindent\textbf{Alberto Verjovsky}\\
Instituto de Matem\'aticas, Unidad Cuernavaca, Universidad Nacional Aut\'onoma de M\'exico.\\
Av. Universidad s/n, Col. Lomas de Chamilpa, C.P. 62210, Cuernavaca, Morelos, M\'exico.\\
\href{mailto:albertoverjovsky@gmail.com}{\nolinkurl{albertoverjovsky@gmail.com}}

\end{document}